\documentclass[preprint,12pt]{elsarticle}
\usepackage{amsmath}
\usepackage{amssymb}
\usepackage{amsthm}
\usepackage{epsfig}

\newtheorem{theorem}{Theorem}
\newtheorem{prop}[theorem]{Proposition}
\newtheorem{lemma}[theorem]{Lemma}
\newtheorem{cor}[theorem]{Corollary}

\newtheorem{definition}{Definition}

\newcommand{\EE}{\mathbb{E}}

\newcommand{\ZZ}{\mathbb{Z}}

\newcommand{\scs}{\textsc{scs}}
\newcommand{\lcs}{\textsc{lcs}}
\newcommand{\uscs}{\underline{\textsc{scs}}}
\newcommand{\ulcs}{\underline{\textsc{lcs}}}
\newcommand{\oscs}{\overline{\textsc{scs}}}
\newcommand{\olcs}{\overline{\textsc{lcs}}}
\newcommand{\sud}{\textsc{Sud}}
\newcommand{\mnc}{\textsc{mnc}}

\DeclareMathOperator{\ch}{ch}

\journal{Discrete Mathematics}
\begin{document}

\begin{frontmatter}
\title{Critical Sets for Sudoku and General Graphs}
\author{Joshua Cooper\corref{cor1}\fnref{label2}}
\ead{cooper@math.sc.edu}
\ead[url]{http://www.math.sc.edu/~cooper/}
\fntext[label2]{This work was funded in part by NSF grant DMS-1001370.}
\cortext[cor1]{Department of Mathematics, University of South Carolina, 1523 Greene St., Columbia SC 29210}
\address{1523 Greene St., Columbia, SC 29210}
\author{Anna Kirkpatrick\corref{cor1}\corref{cor2}}
\ead{kirkpate@email.sc.edu}
\address{1523 Greene St., Columbia, SC 29210.}
\cortext[cor2]{Corresponding author: +1 864 349 6669}
\begin{abstract}
We discuss the problem of finding {\it critical sets} in graphs, a concept which has appeared in a number of guises in the combinatorics and graph theory literature.  The case of the Sudoku graph receives particular attention, because critical sets correspond to minimal fair puzzles.  We define four parameters associated with the sizes of extremal critical sets and (a) prove several general results about these parameters' properties, including their computational intractability, (b) compute their values exactly for some classes of graphs, (c) obtain bounds for generalized Sudoku graphs, and (d) offer a number of open questions regarding critical sets and the aforementioned parameters.
\end{abstract}

\end{frontmatter}

A recent announcement due to McGuire, et al. (\cite{CMT12}), surprised many in the community of Sudoku researchers, amateur and professional alike, with its complete resolution of the ``minimum number of clues'' (\mnc) problem.  {\em Sudoku} is a single-player game in which one completes a partial $9 \times 9$ matrix $M$ all of whose entries are drawn from $\{1,2,3,4,5,6,7,8,9\}$ by appealing to the {\em rules}: no number may appear twice in any row, any column, or any of the nine ``blocks'', each a $3 \times 3$ submatrix with indices $\{3a+1,3a+2,3a+3\} \times \{3b+1,3b+2,3b+3\}$ for $a, b \in \{0,1,2\}$.  A ``board'' is a matrix adhering to these rules; a ``puzzle'' is a partially filled-in board; the nonempty entries of a puzzle are called ``clues'' or ``givens''.  A puzzle is said to be ``fair'' if it can be completed to a valid board in precisely one way.  The \mnc\, problem asks: what is the fewest number of clues in a fair Sudoku puzzle?  While it was long suspected that the answer is $17$, a proof seemed out of reach until \cite{CMT12}.

However, the sense in which the authors of \cite{CMT12} ``proved'' that the solution is indeed $17$ arguably does not meet modern standards of mathematical rigor.  (The paper briefly acknowledges this deficiency, although popular press' wide reporting of the result typically did not address this important, if subtle, issue.)  There are several interesting ideas presented in the aforementioned manuscript -- mostly careful case reductions and very clever search strategies -- but, in the end, the result relied on a {\em year-long} computation, amounting to $7.1$ million core hours on an SGI Altic ICE 8200EX cluster with 320 nodes, each of which consisted of two Intel Xeon E5650 hex-core processors with 24GB of RAM.  Even if one sets aside well-worn (and important) philosophical critiques of computer-assisted proofs that appeal to uncheckability by humans and the social nature of proof, it is almost inconceivable that this enormous computation on an extremely complicated configuration of networked and nested devices did not experience hardware errors (due to manufacturing defects, cosmic rays, background radiation, the inherent stochasticity of quantum mechanics, etc.) {\em and} software errors (bugs in the various operating systems, firmware, algorithmic code, GUIs, etc.).  While such errors may not have produced an incorrect answer, they certainly undermine the definitiveness of the result.  Therefore, we wish to draw attention to the subject matter of ``critical sets'' for graph colorings, a concept that neatly generalizes the \mnc\, problem as well as several other questions scattered throughout the discrete mathematics literature, in the hopes that greater visibility might eventually lead to human-readable solutions to questions like the \mnc.

We begin by defining ``determining sets'' for graph colorings: a set $S$ of vertices with the property that the coloring, restricted to $S$, can be completed in precisely one way (i.e., back to the original coloring).

\begin{definition} A ``determining set'' of vertices in a graph $G = (V,E)$ with respect to a proper vertex coloring $c : V \rightarrow [\chi(G)]$ is a set $S \subseteq V$ with the property that, for any proper vertex coloring $c^\prime$ of $G$, if $c^\prime |_{S} \equiv c |_{S}$, then $c^\prime \equiv c$ on all of $V$.
\end{definition}

\noindent If a determining set is minimal with respect to this property, we call it ``critical''.

\begin{definition} A ``critical set'' of vertices in a graph $G = (V,E)$ with respect to a proper vertex coloring $c : V \rightarrow [\chi(G)]$ is a minimal determining set for the pair $(G,c)$.
\end{definition}

The cardinality of the largest and smallest critical sets in various graphs have appeared in a number of guises.  Latin squares (and, of course, the special subclass of them that comprise Sudoku boards), matching theory, design theory, and the study of dominating vertex sets all feature variants of this idea.  (See \cite{HM06} for a more comprehensive list of related topics and references.)  Therefore, we define the following parameters.

\begin{definition} For a graph $G = (V,E)$ and a vertex coloring $c : V \rightarrow [\chi(G)]$, define
\[
\scs(G,c) = \min \{|X| : X \text{ is a critical set for } (G,c)\}
\]
and
\[
\lcs(G,c) = \max \{|X| : X \text{ is a critical set for } (G,c)\}.
\]
\end{definition}

\begin{definition} For a graph $G = (V,E)$, define
\[
\uscs(G) = \min_{\substack{c : V(G) \rightarrow [\chi(G)] \\ c \text{ proper}}} \scs(G,c)
\]
\[
\ulcs(G) = \min_{\substack{c : V(G) \rightarrow [\chi(G)] \\ c \text{ proper}}} \lcs(G,c).
\]
Similarly, define
\[
\oscs(G) = \max_{\substack{c : V(G) \rightarrow [\chi(G)] \\ c \text{ proper}}} \scs(G,c)
\]
\[
\olcs(G) = \max_{\substack{c : V(G) \rightarrow [\chi(G)] \\ c \text{ proper}}} \lcs(G,c).
\]
\end{definition}

A few words on notation.  For graph-theoretic concepts, unless stated explicitly, we generally rely on the conventions of \cite{D00}; in particular, we sometimes write the edge $\{x,y\}$ simply as $xy$.  For a vertex $v \in V(G)$, we define the ``neighborhood'' of $v$ to be $N_G(v) = \{w \in V(G) : vw \in E(G)\}$, where the subscript may be omitted if it is obvious.  The symbol ``$\Box$'' denotes Cartesian product and ``$\boxtimes$'' the strong graph product.  Given two graphs $G = (U,E)$ and $H = (V,F)$, the vertex sets of $G \Box H$ and $G \boxtimes H$ are both $U \times V$.  The edge set of the former is given by all pairs of the form $\{(u,v),(u^\prime,v^\prime)\}$ with $[(u=u^\prime)\wedge(vv^\prime \in E(H))] \vee [(v=v^\prime)\wedge(uu^\prime \in E(G))]$; the edge set of the latter is given by all pairs of the form $\{(u,v),(u^\prime,v^\prime)\}$ with $[(u=u^\prime)\vee(uu^\prime \in E(G))] \wedge [(v=v^\prime)\vee(vv^\prime \in E(H))]$.  Finally, the set $\{1,\ldots,n\}$ is denoted $[n]$.

\section{General Bounds and Observations}

In this section, we introduce a few important definitions and prove some general facts about critical sets.

\begin{definition} A coloring $c : V(G) \rightarrow [k]$ of a graph $G$ is ``optimal'' if $k = \chi(G)$.
\end{definition}

\begin{definition} A graph $G$ is ``uniquely colorable'' if it has exactly one optimal coloring up to permutation of the colors.
\end{definition}

\begin{definition} A graph $G$ is ``critically $k$-uniform'' if every critical set $S \subset V(G)$ satisfies $|S| = k$; it is ``critically uniform'' if it is critically $k$-uniform for some $k$.
\end{definition}

Note that the condition of $G$ being critically $k$-uniform is equivalent to the statement that
\[
k = \olcs(G) = \ulcs(G) = \oscs(G) = \uscs(G).
\]

\begin{prop} \label{prop:uniquelycolorableimpliessamesize} If a graph $G$ is uniquely colorable, then $G$ is critically $(\chi(G)-1)$-uniform.
\end{prop}
\begin{proof} First, every critical set $S$ must have cardinality at least $\chi(G)-1$; if $|S| < \chi(G)-1$, then, given any coloring $c : V(G) \rightarrow [\chi(G)]$, there are at least two extensions of $c|_S$ to a proper coloring: $c$ itself and $c^\prime$, where
\[
c^\prime(v) = \left \{ \begin{array}{ll} c(v) & \text{ if } c(v) \in c(S) \\ \pi(c(v)) & \text{ if } c(v) \not \in c(S), \end{array} \right .
\]
where $\pi$ is any permutation of $[\chi(G)]$ so that $\pi|_{[\chi(G)] \setminus c(S)}$ is not the identity function.

Then, suppose $S$ is a critical set for the coloring $c : V(G) \rightarrow [\chi(G)]$.  Note that, if $c(v) = c(w)$, then at most one of $v$ or $w$ is in $S$, as the partition into color classes is uniquely determined; therefore, $|S| = |c(S)|$.  If $|S| = |c(S)| = \chi(G)$, then let $S^\prime = S \setminus \{c^{-1}(\chi(G))\}$.  Since the partition $\Pi = \{c^{-1}(1),\ldots,c^{-1}(\chi(G))\}$ is unique, and $\tilde{c} = c|_{S^\prime}$ determines the color of all blocks but one, the last block must be colored $\chi(G)$, and the only proper coloring extending $\tilde{c}$ is $c$ itself. Therefore, $|S| \leq \chi(G)-1$.
\end{proof}

Note that it is not true that, if a graph is critically uniform, then it is uniquely colorable, as evidenced by the graph obtained by joining a pendant edge to each vertex of a $K_3$.  (In this case, all four parameters equal $4$.)  However, we have been unable to determine the status of the full converse of Proposition \ref{prop:uniquelycolorableimpliessamesize}: whether a graph being $(\chi(G)-1)$-uniform implies unique colorability.

\begin{cor} \label{cor:bipartite} If $G$ is bipartite and consists of $k$ components, then $G$ is critically $k$-uniform.
\end{cor}
\begin{proof} Since $\chi(G) = 2$, every proper coloring $c$ gives the same bipartition of the vertex set of each component.  Therefore, specifying the color of one vertex of each component completely determines the rest, i.e., a set $S \subseteq V(G)$ is determining if it contains a vertex of each component.  Furthermore, one must specify the color of at least one vertex of each component, or else the colors within the omitted component could be swapped in any extension of $c|_S$ to a proper coloring.  The only critical sets in $G$, then, have cardinality $k$.
\end{proof}

\begin{prop} For any graph $G$, $\uscs(G) \leq \oscs(G)$, $\uscs(G) \leq \ulcs(G)$, $\ulcs(G) \leq \olcs(G)$, and $\oscs(G) \leq \olcs(G)$.
\end{prop}
\begin{proof} Obvious.
\end{proof}

\begin{prop} \label{prop:cantbeall} For any graph $G$, every critical set $S$ satisfies $|S| \leq |V(G)|-1$.
\end{prop}
\begin{proof} Suppose $t = \chi(G)$ and $V(G)$ is a critical set for some coloring $c : V(G) \rightarrow [\chi(G)]$.  Let $\mathcal{N}(v) = c(N(v))$.  Note that, for each vertex $v \in V(G)$, $|\mathcal{N}(v)| \leq t - 2$.  Therefore, for each $v \in c^{-1}(\chi(G))$, the set $\mathcal{C}(v) = [\chi(G)-1] \setminus \mathcal{N}(v)$ is nonempty; let $\alpha_v$ be an arbitrary element of $\mathcal{C}(v)$, and define a new coloring $c^\prime : V(G) \rightarrow [\chi(G)]$ by
\[
c^\prime(v) = \left \{ \begin{array}{ll} c(v) & \text{ if } c(v) \in [\chi(G)-1] \\ \alpha_v & \text{ otherwise.} \end{array} \right .
\]
Note that $\chi(G) \not \in c^\prime(V(G))$.  Furthermore, $c^{-1}(\chi(G))$ is an independent set, so $c^\prime$ is a proper coloring of $G$.  Therefore, $c^\prime$ is actually a $(t-1)$-coloring of $G$, contradicting the definition of $t$.
\end{proof}

\section{Specific Graphs}

For even cycles, Corollary \ref{cor:bipartite} gives the values of the four parameters.

\begin{cor} For $n$ even, $C_n$ is critically $1$-uniform.
\end{cor}

We need only consider the case of $n$ odd now.  Throughout the sequel, we label the vertices of $C_n$ in cyclic order as $v_0,\ldots,v_{n-1}$; the indices are interpreted $\bmod{\,n}$.

\begin{theorem} For $n$ odd,
\[
\uscs(C_n) = \frac{n+1}{2}.
\]
\end{theorem}
\begin{proof}
Since $C_n$ is not bipartite, any determining set $S$ cannot exclude any two consecutive vertices.  Indeed, suppose $v_0, v_1 \not \in S$.  Fix a proper vertex coloring $c : V \rightarrow \{0,1,2\}$.  If $c(v_{n-1}) = c(v_2)$ ($=0$ without loss of generality), then there are two ways to color $v_0$ and $v_1$: $c(v_0) = 1$ and $c(v_1) = 2$ or $c(v_0) = 2$ and $c(v_1) = 1$.  If $c(v_{n-1}) \neq c(v_2)$, without loss of generality $c(v_{n-1}) = 0$ and $c(v_2) = 1$, then either $c(v_0) = 1$ and $c(v_1)=2$ or $c(v_0)=2$ and $c(v_1)=0$.  In either case, there is more than one way to complete $c|_S$ to a coloring, a contradiction.

Now, if $|S| < n/2$, then $S$ excludes two consecutive vertices, a contradiction.  Therefore, $|S| \geq n/2$, i.e., $|S| \geq (n+1)/2$ (since $|S| \in \ZZ$).  On the other hand, we must exhibit $3$-colorings of $C_n$, $n$ odd, along with sets of $(n+1)/2$ vertices which determine the coloring.  We consider three cases:
\begin{description}
\item{$n \equiv 0 \pmod 3$ :} Let $c(v_j) = j \pmod 3$, and $S = \{v_0,v_2,v_4,\ldots,v_{n-1}\}$.  Then, since $c(v_i) \neq c(v_{i+2})$ for each $i \in \mathbb{Z}_n$, the color of each vertex in $V \setminus S = \{v_1,\ldots,v_{n-2}\}$ is uniquely determined.
\item{$n \equiv 1 \pmod 3$ :} Let $c(v_j) = j \pmod 3$ for $j < n-1$ and $c(v_{n-1}) = 1$, and $S = \{v_0$, $v_2$, $v_4$, $\ldots$, $v_{n-5}$, $v_{n-3}$, $v_{n-2}\}$.  Then, since $c(v_i) \neq c(v_{i+2})$ for each $i \in \{0,\ldots,n-4\}$ and $c(v_{n-2}) = 2 \neq 0 = c(v_0)$, the color of each vertex in $V \setminus S = \{v_1,\ldots,v_{n-6},v_{n-4},v_{n-1}\}$ is uniquely determined.
\item{$n \equiv 2 \pmod 3$ :} Let $c(v_j) = j \pmod 3$, and $S = \{v_0,v_2,v_4,\ldots,v_{n-5},v_{n-3},v_{n-1}\}$.  Then, since $c(v_i) \neq c(v_{i+2})$ for each $i \in \{0,\ldots,n-2\}$, the color of each vertex in $V \setminus S = \{v_1$, $\ldots$, $v_{n-6}$, $v_{n-4}$, $v_{n-2}\}$ is uniquely determined.
\end{description}
\end{proof}

\begin{theorem} For $n$ odd,
\[
\olcs(C_n) = n-1.
\]
\end{theorem}
\begin{proof} Color $C_n$ by
\[
c_0(v_i) = \left \{ \begin{array}{ll} i \pmod 2 & \textrm{ if } 0 \leq i < n-1 \\ 2 & \textrm{ if } i=n-1. \end{array} \right .
\]
Let $S$ consist of all vertices but $v_{n-1}$.  Then $S$ determines the coloring, since $v_{n-1}$ has one $1$-colored vertex $v_{n-2}$ and one $0$-colored vertex $v_0$ from $S$ adjacent to it.  On the other hand, it is not possible to remove a vertex from $S$ and have it still determine the coloring: removing either $v_0$ or $v_{n-2}$ would leave a gap of two adjacent vertices, which we noted above is not possible in a determining set, and removing any other vertex allows for a recoloring of that vertex with $2$.  Therefore, $S$ is critical, and $\olcs(C_n) \geq n-1$.  To conclude that $\olcs(C_n) \leq n-1$, we need only invoke Proposition \ref{prop:cantbeall}.
\end{proof}

\begin{theorem} For $n$ odd,
\[
\oscs(C_n) = n-2.
\]
\end{theorem}
\begin{proof} Consider any critical set $S$ for the coloring $c_0$ considered in the preceding proof.  If $v_i \in S$, then $v_{i+1} \in S$ or $v_{i+2} \in S$, because $V(C_n) \setminus S$ cannot have two consecutive vertices.  If $v_{i+1} \not \in S$ and $i \leq n-4$, then $c(v_{i+2}) = c(v_i)$, so $v_{i+1} \in S$ or else it could be recolored with the single element of $\ZZ_3 \setminus \{c(v_i),c(v_{i+1})\}$.  Therefore, $v_i \in S$ for all $1 \leq i \leq n-3$, and there are only three vertices which could be omitted from $S$: $v_{n-2}$, $v_{n-1}$, and $v_0$.  All three cannot be omitted, or else $V(C_n) \setminus S$ would contain two consecutive vertices.  Therefore, $|S| \geq n-2$ and $\oscs(C_n) \geq n-2$.
\end{proof}

For the next few proofs, we find the following definition useful.

\begin{definition} A vertex $v$ of a graph $G$ associated with a coloring $c : V(G) \rightarrow [k]$ is ``colorful'' if $|c(\{v\} \cup N(v))| = k$.
\end{definition}

\begin{lemma} Every proper coloring of $C_n$, for $n$ odd, admits a colorful vertex.
\end{lemma}
\begin{proof}
Let $c : V(C_n) \rightarrow \ZZ_3$ be any proper vertex coloring of $C_n$.  Towards a contradiction, suppose that every $v_i \in C_n$, has the property that $|\{c(v_{i-1}),c(v_{i}),c(v_{i+1})\}| = 2$.  Without loss of generality, $c(v_1) = 0$ and $c(v_2) = 1$, so $c(v_3) = 0$, or else $v_2$ would yield a contradiction.  But then $c(v_4) = 1$, or else $v_3$ would yield a contradiction.  Proceeding around the cycle, we see that the colors alternate, which is not possible for an odd cycle.  Therefore, every coloring has a colorful vertex $v$.
\end{proof}

Note that no critical set can contain a colorful vertex and all of its neighbors, as its neighbors already determine that vertex's color.

\begin{theorem} For $n$ odd,
\[
\ulcs(C_n) = \left \{ \begin{array}{ll} \frac{n+3}{2} & \textrm{ if } n \equiv 1 \pmod{4} \\ \frac{n+1}{2} & \textrm{ if } n \equiv 3 \pmod{4}. \end{array} \right .
\]
\end{theorem}
\begin{proof} Color $V(C_n)$ as follows:
\[
c(v_i) = \left \{ \begin{array}{ll} 0 & \textrm{if } 2 | i \textrm{ and } i < n-1 \\ 1 & \textrm{if } i \equiv 1 \pmod 4 \\ 2 & \textrm{if } i \equiv 3 \pmod 4 \\ 3-c(v_{n-2}) & \textrm{if } i=n-1 \end{array} \right .
\]
It is easy to check that this is a proper coloring.  Let $S$ be a critical set for $c$.  Then $S$ must contain $v_1$, $v_3$, \ldots, $v_{n-4}$, as, even if the vertices on both sides of these $v_i$ are included, the color of $v_i$ is not determined.  However, then it is unnecessary to include $v_2$, $v_4$, \ldots, $v_{n-5}$, since these vertices are colorful.  Now, at most three of the colorful vertices $v_{n-3}$, $v_{n-2}$, $v_{n-1}$, and $v_0$ can be in $S$, or else $S$ would contain three consecutive vertices the middle one of which is colorful.  Therefore,
\[
|S| \leq \frac{n-3}{2} + 3 = \frac{n+3}{2}.
\]
Therefore, $\ulcs(C_n) \leq (n+3)/2$.  If $n \equiv 3 \pmod{4}$, we claim that in fact, $S$ can include at most two of the vertices $v_{n-3}$, $v_{n-2}$, $v_{n-1}$, and $v_0$.  Note that $c(v_{n-3}) = 0$, $c(v_{n-2}) = 1$, $c(v_{n-1}) = 2$, and $c(v_0) = 0$, and {\em all four of these vertices are colorful}.  Suppose $S$ contained three of these vertices.  Since $v_{n-4} \in S$ and $v_1 \in S$, either $S$ contains $v_{n-4}$, $v_{n-3}$, and $v_{n-2}$, or it contains $v_{n-1}$, $v_0$, and $v_1$; in either case, the middle vertex is determined by the colors of its neighbors, a contradiction.  We may conclude that $|S| \leq (n-3)/2 + 2 = (n+1)/2$.  Note that $(n+1)/2 = \uscs(C_n) \leq \ulcs(C_n)$.  Therefore, if $n \equiv 3 \pmod{4}$, $\ulcs(C_n) = (n+1)/2$.

If $n \equiv 1 \pmod{4}$, we claim that every coloring in fact admits a critical set of size at least $(n+3)/2$.  Suppose not.  Then, by the above bounds, every critical set $S$ for some coloring $c$ is size $(n+1)/2$.  Since $S$ cannot omit two consecutive vertices, $S$ must contain every other vertex all the way around the cycle, except for exactly one location where it contains two consecutive vertices.  Without loss of generality, we may assume these are $v_0$ and $v_1$, so $S = \{v_0,v_1,v_3,v_5,\ldots\}$.  Note that $v_k$ must be colorful for $k > 0$ even, or else it would have to be included in $S$.  Furthermore, every $v_k$ with $k > 1$ odd must be noncolorful, or else we could replace $v_k$ with $v_{k-1}$ and $v_{k+1}$ in $S$, obtaining a critical set of size $> (n+1)/2$.  Therefore, if $v_1$ is noncolorful, then, without loss of generality, $c(v_0) = 0$, $c(v_1) = 1$, $c(v_2)=0$, $c(v_3)=2$, $c(v_4)=0$, $c(v_5)=1$, $c(v_6)=0$, $c(v_7)=2$, and so on.  (Specifically, $c(v_k) = 0$ if $k$ is even, $c(v_k) = 1$ if $k \equiv 1 \pmod{4}$ and $c(v_k)=2$ if $k \equiv 3 \pmod{4}$.)  However, then $c(v_{n-1}) = 0$, contradicting the fact that $c(v_0)=0$ and $c$ is a proper coloring.  A similar argument applies to $v_0$ being noncolorful by reversing the indices of $C_n$ via the map $j \mapsto 1-j$.  If $v_0$ and $v_1$ are colorful, then, without loss of generality, $c(v_0) = 0$, $c(v_1)=2$, $c(v_2)=1$, $c(v_3)=0$, $c(v_4)=2$, $c(v_5)=0$, $c(v_6)=1$, and so on.  (Specifically, for $2 \leq k \leq n-1$, $c(v_k) = 0$ if $k$ is odd, $c(v_k) = 1$ if $k \equiv 2 \pmod{4}$ and $c(v_k)=2$ if $k \equiv 0 \pmod{4}$.)  However, since $n-1 \equiv 0 \pmod{4}$, $c(v_{n-1}) = 2$, contradicting the fact that $v_0$ is colorful.  Therefore, if $n \equiv 1 \pmod{4}$, then $\ulcs(C_n) = (n+3)/2$.
\end{proof}

For the next result, let $\sud_n$ denote the Sudoku graph, i.e., $(K_{n^2} \Box K_{n^2}) \cup (nK_n \boxtimes nK_n)$. Our proof draws heavily on that of a similar statement for Latin squares occurring in \cite{GHM05}.

\begin{theorem}
\[
\oscs(\sud_n) = n^4 - \Omega(n^{10/3}).
\]
\end{theorem}
\begin{proof}
For each of the $n^4$ vertices $(a,b) \in [n^2] \times [n^2]$ of $\sud_n$, let $t_{(a,b)}$ denote a uniform random real number in the interval $[0,1]$. Fix a proper vertex coloring $\phi : V(\sud_n) \rightarrow [n^2]$.  We construct a determining set $S$ by the following random process:
\begin{enumerate}
\item Let $j=0$ and $S_0 = [n^2] \times [n^2]$, and start a timer at $t = 0$.
\item \label{nextvertex} Wait until the next vertex $(a,b) \in V(\sud_n)$ is ``born'' at time $t_{(a,b)}$ or else $t=1$, in which case go to (\ref{endstep}).
\item If, for all $\gamma \in [n^2]$, there exists a $v \in V(\sud_n)$ so that $\phi(v) = \gamma$ and $v \in S_j$ and $\{v,(a,b)\} \in E(\sud_n)$, then $S_{j+1} = S_j \setminus \{(a,b)\}$.
\item $j \leftarrow j+1$
\item If $t < 1$, go to (\ref{nextvertex}).
\item \label{endstep} Return $S_{n^2}$.
\end{enumerate}
Note that the algorithm executes properly with probability $1$, since almost surely $t_{v} \neq t_{w}$ for any $v, w \in V(\sud_n)$ with $v \neq w$.  Also, it is easy to see that $S_{n^2}$ is a determining set for $\phi$.

Now, we compute $\EE[|S_{n^2}|]$.  When a vertex $v = (a,b)$ is born at time $t$ to be considered for omission from $S_j$ (to obtain $S_{j+1}$), the probability that any particular vertex $w$ in its neighborhood is not contained in $S_j$ is bounded above by the probability that it has been born, i.e., $t$.  Fix $\gamma \in [n^2]$, $\gamma \neq \phi(v)$.  There is exactly one vertex $c_v^\gamma = (\alpha,\beta) \in V(\sud_n)$ with $a=\alpha$ and $b \neq \beta$ and $\phi(c_v^\gamma) = \gamma$; there is exactly one vertex $r_v^\gamma = (\alpha,\beta) \in V(\sud_n)$ with $a \neq \alpha$ and $b = \beta$ and $\phi(r_v^\gamma) = \gamma$; there is exactly one vertex $b_v^\gamma = (\alpha,\beta) \in V(\sud_n)$ with $\lceil \alpha/n \rceil = \lceil a/n \rceil$ and $\lceil \beta/n \rceil = \lceil b/n \rceil$ and $\phi(b_v^\gamma) = \gamma$.  Let $N_v^\gamma = \{c_v^\gamma,r_v^\gamma,b_v^\gamma\}$; $|N_v^\gamma| = 2$ for exactly $2n-1$ colors $\gamma$ and $|N_v^\gamma| = 3$ for exactly $n^2 - 2n + 1 = (n-1)^2$ colors $\gamma$.  For any particular $\gamma$, the probability that every element of $N_v^\gamma$ has been removed from $S_j$ is at most $1-t^{|N_v^\gamma|}$.  Therefore, the probability that, for every $\gamma \in [n^2] \setminus \{\phi(v)\}$, $S_j \cap N_v^\gamma \neq \emptyset$ is at least
\[
(1 - t^2)^{2n-1} (1-t^3)^{(n-1)^2},
\]
and
\begin{align*}
\EE[|S_{n^2}|] &\leq n^4 \int_0^1 1 - (1 - t^2)^{2n-1} (1-t^3)^{(n-1)^2} \, dt \\
&= n^4 - n^4 \int_0^1 (1 - t^2)^{2n-1} (1-t^3)^{(n-1)^2} \, dt \\
&\leq n^4 - n^4 \int_0^{n^{-2/3}} (1 - t^2)^{2n-1} (1-t^3)^{(n-1)^2} \, dt \\
&\leq n^4 - n^{4-2/3} (1 - n^{-4/3})^{2n-1} (1-n^{-2})^{(n-1)^2} \\
&= n^4 - n^{10/3} \exp \left (- O(n^{-4/3} n-n^{-2}n^2) \right )\\
&= n^4 - n^{10/3} \exp \left (1 - O(n^{-1/3}) \right )\\
&= n^4 - O(n^{10/3}),
\end{align*}
where the constant implicit in the `$O$' is universal.  Thus, for every optimal coloring $\phi$ of $\sud_n$, there exists a determining set of size at most $n^4 - O(n^{10/3})$, and therefore also a critical set of size at most $n^4 - O(n^{10/3})$.  In particular, then,
\[
\oscs(\sud_n) = n^4 - \Omega(n^{10/3}).
\]
\end{proof}

We conclude this section with a table (Figure \ref{fig:fig1}) of the four parameters and their values on all nonbipartite graphs on $5$ vertices.  Every graph on at most $4$ vertices is critically uniform, with parameters: $0$ for $K_1$, $2K_1$, $3K_1$, and $4K_1$; $1$ for $K_2$, $P_3$, $K_{1,3}$, $P_4$, and $C_4$; $2$ for $K_1 \cup K_2$, $K_3$, $2K_2$, $K_1 \cup P_3$, $\overline{K_1 \cup P_3}$, and $\overline{2K_1 \cup K_2}$; and $3$ for $2K_1 \cup K_2$, $K_1 \cup K_3$, and $K_4$.  (The symbols $+e$, $\overline{\cdot}$, and $\cup$ denote the addition of a pendant edge, complementation, and disjoint union, respectively.)

\begin{figure}[ht]
\begin{center}
\begin{tabular}{c}
\includegraphics[width=5in]{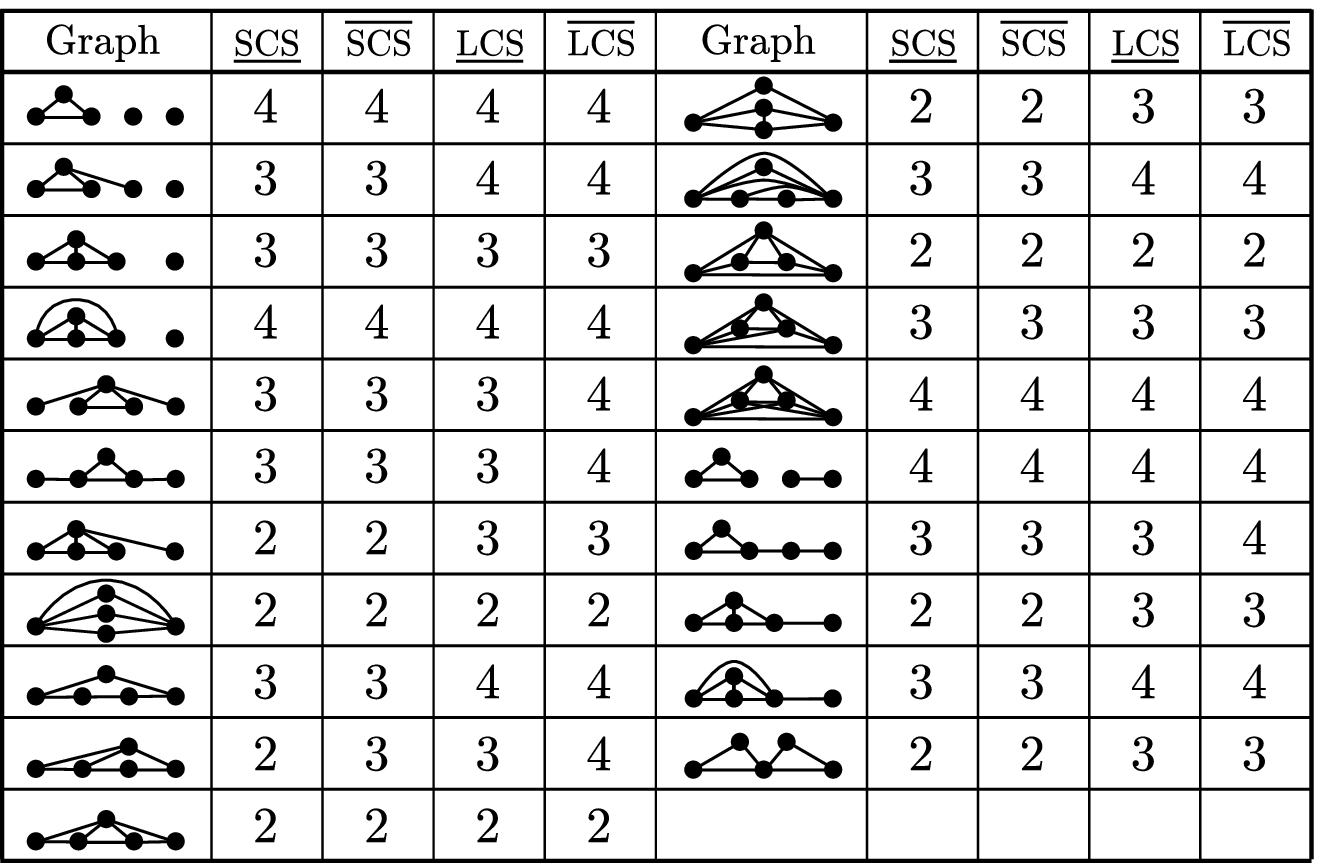}
\end{tabular}
\caption{\label{fig:fig1} A table of the four parameters for all five vertex, nonbipartite graphs.}
\end{center}
\end{figure}

\section{Complexity}

Another important question in studying the graph parameters above is the determining of the difficulty of computing them.  In fact, some aspects of this question have been considered in the recent literature.  Indeed, under the rubric of ``defining sets'', the following two results are almost immediate from the work of Hatami-Tusserkani (\cite{HT12}):

\begin{theorem} It is \textsc{np}-hard to decide, for a given integer $k$ and graph $G$, whether $\uscs(G) \geq k$.
\end{theorem}

\begin{theorem} It is \textsc{np}-hard to decide, for a given integer $k$ and graph $G$, whether $\oscs(G) \geq k$.
\end{theorem}

To show that the ``other'' two parameters are indeed hard to compute, we follow Hatami-Tusserkani's proofs closely and make a few necessary modifications.  Because the constructions are somewhat difficult to relate via text, we include Figure \ref{fig:fig2} for illustration.

\begin{figure}[ht]
\begin{center}
\begin{tabular}{c}
\includegraphics[width=4in]{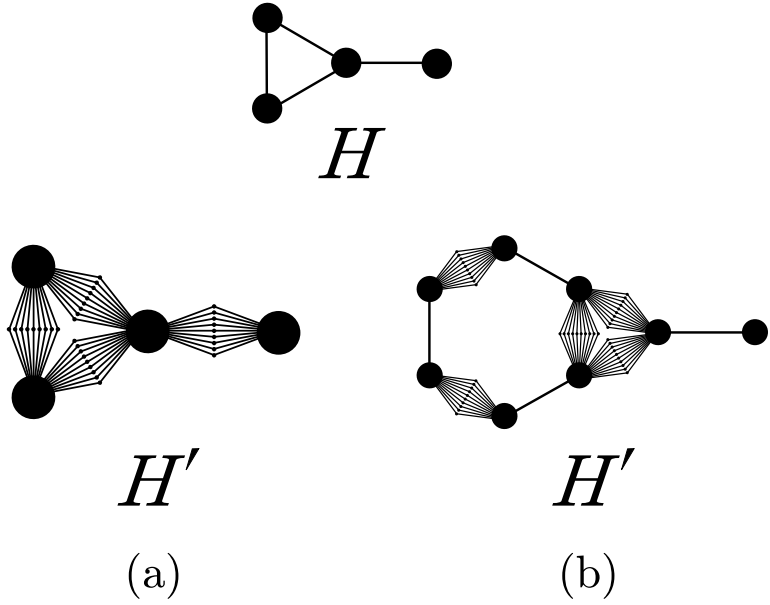}
\end{tabular}
\caption{The constructions appearing in the proofs of Theorems \ref{thm:ulcshard} and \ref{thm:olcshard}.\label{fig:fig2}}
\end{center}
\end{figure}

\begin{theorem} \label{thm:ulcshard} It is \textsc{np}-hard to decide, for a given integer $k$ and graph $G$, whether $\ulcs(G) \geq k$.
\end{theorem}
\begin{proof} Suppose $H$ is a graph with $n$ vertices and $m$ edges.  We show that determining whether or not $H$ is $3$-colorable is reducible to deciding whether $\ulcs(G) \geq k$ for $k = m + n + 3$ and some graph $G$, $G$ to be defined below (but depending in size only polynomially on $n$ and $m$).  We construct a graph $H^\prime$ as follows: $V(H^\prime) = V_1 \cup V_2$, where $V_1 = V(H)$ and $V_2 = E(H) \times [m+n+1]$, and $E(H)$ consists of all pairs of the form $\{v,(e,j)\}$ where $v \in V_1$, $(e,j) \in V_2$, and $v \in e$.  Then let $G = H^\prime \cup K_3$, and denote the set of vertices arising from the $K_3$ as $V_3 = \{x_1,x_2,x_3\}$.  We claim that $\ulcs(G) \geq k$ if and only if $H$ is not $3$-colorable.  If $H$ is $3$-colorable, then let $c : V(H) \rightarrow [3]$ be such a coloring, and define a $3$-coloring of $G$ by setting, for each $v \in V(G)$,
\[
c^\prime(v) = \left \{ \begin{array}{ll} c(v) & \text{ if } v \in V_1 \\ \xi & \text{ if } v = (xy,j) \in V_2 \text{ where } \xi \not \in \{c(x),c(y)\} \\ j & \text{ if } v = x_j. \end{array} \right .
\]
Then every critical set $S$ for $c^\prime$ has cardinality $< k$, because
\begin{enumerate}
\item No two colorful vertices whose neighborhoods are identical are contained in the same critical set together.  Therefore, if $S$ contains a vertex $(e,j) \in V_2$, then it contains none of the vertices $(e,j^\prime)$ with $j \neq j^\prime$.  This allows for at most $m$ vertices from $V_2$ in $S$.
\item $S$ will not contain all three vertices of $V_3$.
\end{enumerate}
Therefore, $|S| \leq m + n + 2 < k$, and $\ulcs(G) < k$.

In the other direction, suppose $H$ is not $3$-colorable.  Therefore, every $3$-coloring of $H$ contains two adjacent vertices of the same color; in particular, any $3$-coloring of $G$ will have two vertices $v$ and $w$ from $V_1$ so that $vw \in E(H)$ to which it assigns the same color.  One must therefore include every one of the $m+n+1$ vertices $(vw,j)$ in any critical set $S$.  Additionally, $S$ contains $2$ elements of $V_3$, so $|S| \geq m+n+1+2 = k$.
\end{proof}

\begin{theorem} \label{thm:olcshard} It is \textsc{np}-hard to decide, for a given integer $k$ and graph $G$, whether $\olcs(G) \geq k$.
\end{theorem}
\begin{proof} Suppose $H$ is a graph with $n$ vertices and $m$ edges.  We show that determining whether or not $H$ is $3$-colorable is reducible to deciding whether $\olcs(G) \geq k$ for $k = (2m+2) \sum_{v \in V(H)} \binom{\deg(v)}{2} + 2$ and some graph $G$, $G$ to be defined below (but depending in size only polynomially on $n$ and $m$).  We define a graph $H^\prime$ as follows: $V(H^\prime)$ consists of vertices $V_1$ of the form $x_{v,e}$, where $v \in V(H)$ and $v \in e \in E(H)$ and vertices $V_2$ of the form $y_{v,e,f,j}$ with $v \in V(H)$, $v \in e \in E(H)$, $v \in f \in E(H)$, $e \neq f$, and $j \in [2m+2]$; $E(H^\prime)$ consists of edges of the form $x_{v,vw}x_{w,vw}$ for each $v, w \in V(H)$, $v \neq w$ and edges of the form $x_{v,e}y_{v,e,f,j}$ for $v \in V(H)$, $v \in e \in E(H)$, $v \in f \in E(H)$, and $j \in [2m+2]$.  Define $G = H^\prime \cup K_3$, and denote the vertices of $K_3$ by $V_3 = \{z_1,z_2,z_3\}$.  We claim that $H$ is $3$-colorable if and only if $\olcs(G) \geq k$.  Suppose $H$ is $3$-colorable; let $c$ be such a coloring, and define, for $x \in V(G)$
\[
c^\prime(x) = \left \{ \begin{array}{ll} c(v) & \text{ if } x = x_{v,e} \in V_1 \\ \xi & \text{ if } x = x_{v,e,f,j} \in V_2 \text{ where } \xi = \min([3] \setminus \{c(v)\}) \\ j & \text{ if } x = z_j. \end{array} \right .
\]
Let $S$ be a critical set for $c^\prime$.  Note that $V_2 \subset S$, because the neighbors of each $y_{v,e,f,j}$ receive only one color.  Furthermore, $S$ must contain at least two vertices from $V_2$, so
\[
|S| \geq (2m+2) \sum_{v \in V(H)} \binom{\deg(v)}{2} + 2 = k.
\]
Suppose $H$ is not $3$-colorable.  Then, for every $3$-coloring of $G$, there are two vertices $x_{v,e}$ and $x_{v,f}$, $e \neq f$, which receive distinct colors.  Then $S$ contains at most one of the vertices $y_{v,e,f,j}$ (since these vertices are colorful and have identical neighborhoods) and at most $2$ vertices from $V_3$, so
\[
|S| \leq (2m+2) \sum_{v \in V(H)} \binom{\deg(v)}{2} + 2m + 2 - (2m + 1) = k - 1.
\]
Therefore, $\olcs(G) < k$.
\end{proof}

\section{Conclusion and Questions}

We conclude with several open questions we find interesting, some of which have appeared in other guises in the literature.  Of course, a central problem is to find a human-readable proof of the (seeming) fact that $\uscs(\sud_3) = 17$, as well as an understanding of the growth rate of $\uscs(\sud_k)$.

\begin{enumerate}
\item Is the converse of Proposition \ref{prop:uniquelycolorableimpliessamesize} true?  That is, if every critical set of a graph $G$ has cardinality $\chi(G)-1$, must $G$ be uniquely colorable?  We have verified computationally that this statement is indeed true for every graph on at most $8$ vertices.
\item What is the list-chromatic number of $\sud_3$?  This is the smallest $k$ so that, if every vertex of $\sud_3$ is associated with a list of $k$ distinct colors, there will be an assignment of colors to vertices from their associated lists which is a proper coloring.  Clearly, $\ch(\sud_3) \geq 9$, but is it equal to $9$?
\item One can define parametrized versions of the four functions $\uscs$, $\oscs$, $\ulcs$, $\olcs$, by asking for the smallest/largest over all proper $k$-colorings -- $k$ not necessarily equal to the chromatic number -- of the size of the smallest/largest critical set for that coloring.  Are these parameters all monotone nondecreasing in $k$?
\item Is there a simple characterization of critically uniform graphs?  By doing an extensive computation, we have verified that there are many graphs with this property which are {\em not} uniquely colorable.
\item What are $\uscs(G)$, $\oscs(G)$, $\ulcs(G)$, and $\olcs(G)$ for $G = K_n \Box K_n$, i.e., Latin squares?  A superlinear lower bound on $\uscs(K_n \Box K_n)$ was obtained only relatively recently (\cite{C07}), for example.
\end{enumerate}

\section{Acknowledgements}

\noindent The authors would like to thank the South Carolina Governor's School for Science and Mathematics, the University of South Carolina Office of Undergraduate Research and the Magellan Scholar Program, and the National Science Foundation for their support, and Chris Boyd and Nick Smith for their valuable insights and assistance.

\bibliographystyle{model1a-num-names}

\end{document}